\documentclass[amstex,12pt, amssymb]{article}

\usepackage{mathtext}
\usepackage[cp1251]{inputenc}
\usepackage[T2A]{fontenc}
\usepackage[dvips]{graphicx}
\usepackage{amsmath}
\usepackage{amssymb}
\usepackage{amsxtra}
\usepackage{latexsym}
\usepackage{ifthen}

\newtheorem{theorem}{Theorem}[section]
\newtheorem{lemma}{Lemma}[section]
\newtheorem{corollary}{Corollary}[section]
\newtheorem{proposition}{Proposition}[section]

\newtheorem{remark}{Remark}[section]

\numberwithin{equation}{section}

\def\Xint#1{\mathchoice
   {\XXint\displaystyle\textstyle{#1}}%
   {\XXint\textstyle\scriptstyle{#1}}%
   {\XXint\scriptstyle\scriptscriptstyle{#1}}%
   {\XXint\scriptscriptstyle\scriptscriptstyle{#1}}%
   \!\int}
\def\XXint#1#2#3{{\setbox0=\hbox{$#1{#2#3}{\int}$}
     \vcenter{\hbox{$#2#3$}}\kern-.5\wd0}}
\def\dashint{\Xint-}

\begin{document}

\markboth{\centerline{Vladimir Ryazanov, Sergei Volkov}}
{\centerline{On Sobolev's mappings on Riemann surfaces}}

\author{{Vladimir Ryazanov, Sergei Volkov}}

\title{{\bf On Sobolev's mappings on Riemann surfaces}}

\maketitle

\large \begin{abstract} In terms of dilatations, it is proved a
series of criteria for continuous and homeomorphic extension to the
boundary of mappings with finite distortion between regular domains
on the Riemann surfaces
\end{abstract}

\bigskip
{\bf 2010 Mathematics Subject Classification: Primary   31A05,
31A20, 31A25, 31B25, 35Q15; Se\-con\-da\-ry 30E25, 31C05, 34M50,
35F45}

\large
%\cc

\section{Introduction}

Recall that {\bf $n-$dimensional topological manifold
$\mathbb{M}^{n}$} is a Hausdorff to\-po\-lo\-gi\-cal space with a
countable base every point of which has an open neighborhood that is
homeomorphic to $\mathbb{R}^{n}$ or, the same, to an open ball in
$\mathbb{R}^{n}$, see e.g. \cite{Fo}. A {\bf chart on the manifold
$\mathbb{M}^{n}$} is a pair $(U,\ g)$ where $U$ is an open subset of
the space $\mathbb{M}^{n}$ and $g$ is a homeomorphism of $U$ on an
open subset of the coordinate space $\mathbb{R}^n$. Note that
$\mathbb{R}^2$ is homeomorphic to $\mathbb{C}$ through the
correspondence $(x,y)\Rightarrow z:\ =x+iy$.

A {\bf complex chart} on the two-dimensional manifold $\mathbb{S}$
is a homeomorphism $g$ of an open set $U\subseteq\ \mathbb{S}$ onto
an open set $V\subseteq\ \mathbb{C}$ under that every point $p \in
U$ corresponds a number $z$, its {\bf local coordinate}. The set $U$
itself is sometimes called a chart. Two complex charts $g_1:U_1\to
V_1$ and $g_2:U_2\to V_2$ are called {\bf conformal confirmed} if
the map
\begin{equation} \label{eq1}
g_2\circ g_1^{-1}:\ \ g_1(U_1\cap U_2)\ \to\ g_2(U_1\cap U_2)
\end{equation} is conformal. A {\bf complex atlas} on $\mathbb{S}$
is a collection of mutually conformal confirmed charts covering
$\mathbb{S}$. Complex atlases on $\mathbb{S}$ are called conformal
confirmed if their charts are so.

A {\bf complex structure} on a two-dimensional manifold $\mathbb{S}$
is an equivalence class of conformal confirmed atlases on
$\mathbb{S}$. It is clear that a complex structure on $\mathbb{S}$
can be determined by one of its atlases. Moreover, uniting all
atlases of a complex structure on $\mathbb{S}$, we obtain its atlas
$\Sigma$ that is maximal by inclusion. Thus, the complex structure
can be identified with its maximal atlas $\Sigma$. The {\bf
conjugate complex structure} $\overline{\Sigma}$ on $\mathbb{S}$
consists of the charts $\bar{g}$ of the complex conjugation of $g\in
\Sigma$ that connected each to other by the anti-conformal mapping
of $\mathbb{C}$ of the mirror reflection with respect to the real
axis not keeping orientation. Thus, we have no uniqueness for the
complex structures on two-dimensional manifolds.

A {\bf Riemann surface} is a pair $(\mathbb{S},\Sigma) $ consisting
of a two-dimensional manifold $\mathbb{S}$ and a complex structure
$\Sigma$ on $\mathbb{S}$. As usual, it is written only $\mathbb{S}$
instead of $(\mathbb{S},\Sigma) $ if the choice of the complex
structure $\Sigma$ is clear by a context. Given a Riemann surface
$\mathbb{S}$, a {\bf chart} on $\mathbb{S}$ is a complex chart in
the maximal atlas of its complex structure.

Now, let $\mathbb{S}$ and $\mathbb{S}^*$ be Riemann surfaces. We say
that a mapping $f:\mathbb{S}\to\mathbb{S}^*$ belongs to the Sobolev
class $W^{1,1}_{\rm loc}$ if $f$ belongs to $W^{1,1}_{\rm loc}$ in
local coordinates, i.e., if for every point $p\in\mathbb{S}$ there
exist charts $g: U\to V$ and $g_*: U_*\to V_*$ on $\mathbb{S}$ and
$\mathbb{S}^*$, correspondingly, such that $p\in U$, $f(U)\subseteq
U_*$ and the mapping
\begin{equation} \label{eq2}
F\ :\,=\ g_*\circ f\circ g^{-1}:\ \ V\to V_*
\end{equation}
belongs to the class $W^{1,1}_{\rm loc}$. Note that the latter
property is invariant under re\-pla\-ce\-ments of charts because the
class $W^{1,1}_{\rm loc}$ is invariant with respect to replacements
of variables in $\mathbb{C}$ that are local quasiisometries, see
e.g. Theorem 1.1.7 in \cite{Maz_1985}, and conformal mappings are so
in view of boundedness of their derivatives on compact sets. Note
also that domains $D$ and $D^*$, i.e. open connected sets, on
Riemann surfaces $\mathbb{S}$ and $\mathbb{S}^*$ are themselves
Riemann surfaces with complex structures induced by the complex
structures on $\mathbb{S}$ and $\mathbb{S}^*$,
cor\-res\-pon\-ding\-ly. Hence the definition given above can be
extended to mappings $f:D\to D_*$.

Recall also that functions of the class $W^{1,1}_{\rm loc}$ in
$\mathbb{C}$ are absolutely continuous on lines, see e.g. Theorem
1.1.3 in \cite{Maz_1985}, and, consequently, almost everywhere have
partial derivatives. By the Gehring-Lehto theorem such
complex-valued functions also have almost everywhere the total
differential if they are open mappings, i.e., if they map open sets
onto open sets, see \cite{GL}. Note that this result was before it
obtained by Menshov for homeomorphisms and, moreover, his proof can
be extended to open mappings with no changes, see \cite{Me}. We will
apply this fact just to homeomorphisms. It is clear that the
property of differentiability of mappings at a point is invariant
with respect to replacements of local coordinates on Riemann
surfaces. Note that, under the research of the boundary behavior of
homeomorphisms $f$ between domains on Riemann surfaces, it is
sufficient to be restricted by sense preserving homeomorphisms
because in the case of need we may pass to the conjugate complex
structure in the image.

\bigskip

\section{Definitions and preliminary remarks}

First of all note that by the Uryson theorem topological manifolds
are metrizable because they are Hausdorff regular topological spaces
with a countable base, see \cite{Ur} or Theorem 22.II.1 in
\cite{Ku$_1$}.

\medskip

As well-known, see e.g. Section III.III.2 in \cite{Stoilov}, the
Riemann surfaces are orientable two-dimensional manifolds and,
inversely, orientable two-dimensional manifolds admit complex
structures, i.e., are supports of Riemann surfaces, see e.g. Section
III.III.3 in \cite{Stoilov}, see also Theorem 6.1.9 in \cite{ZVC}.
Moreover, two-dimensional topological manifolds are triangulable,
see e.g. Section III.II.4 in \cite{Stoilov}, see also Theorem 6.1.8
in \cite{ZVC}.

\medskip

Every orientable two-dimensional manifold $\mathbb{S}$ has the {\bf
canonical re\-pre\-sen\-ta\-tion of Kerekjarto-Stoilow} in the form
of a part of the extended complex plane
$\overline{\mathbb{C}}=\mathbb{C}\cup\{\infty\}$ that appears after
removing from ${\mathbb{C}}$ a compact totally disconnected set $B$
of points of the real axis and of a finite or countable collection
of pairs mutually disjoint disks that are symmetric with respect to
the real axis whose boundary circles can be accumulated only to the
set $B$ and whose points pairwise identified, see e.g. III.III in
\cite{Stoilov}. The number $g$ of these pairs of glued circles is
called a {\bf genus of the surface} $\mathbb{S}$.

\medskip

It is clear that the topological model of Kerekjarto-Stoilow is
homeomorphic to the sphere $\mathbb{S}^2\simeq\overline{\mathbb{C}}$
in $\mathbb{R}^3$ with $g$ handles and a compact totally
disconnected set of punctures in $\mathbb{S}^2$. Gluing these
punctures in the Kerekjarto-Stoilow model by points of the set $B$,
we obtain a compact topological space that is not a two-dimensional
manifold if $g=\infty$. Similarly, joining the boundary elements to
the initial surface $\mathbb{S}$, that correspond in a one-to-one
manner to the points of the set $B$, we obtain its {\bf
compactification by Kerekjarto-Stoilow} $\overline{\mathbb{S}}$.

\medskip

Next, let $x_k$, $k=1,2,\ldots$, be a sequence of points in a
topological space $X$. It is said that a point $x_*\in X$ is a {\bf
limit point} of the sequence $x_k$, written
$x_*=\lim\limits_{k\to\infty}\, x_k$ or simply $x_k\to x_*$ if every
neighborhood $U$ of the point $x_*$ contains all points of the
sequence except its finite collection.  Let $\Omega$ and $\Omega_*$
be open sets in topological spaces $X$ and $X_*$, correspondingly.
Later on, $C(x,f)$ denotes the {\bf cluster set} of a mapping
$f:\Omega\to \Omega_*$ at a point $x\in\overline \Omega$, i.e.,
\begin{equation}\label{eq8.5} C(x,f)\ :\ =\ \left\{x_*\in X_*:\ x_*\ =\ \lim\limits_{k\to\infty}f(x_k),\ x_k\to x,\ x_k\in
\Omega\ \right\}\end{equation} It is known that the inclusion
$C(x,f)\subseteq\partial \Omega_*$, $x\in\partial\Omega$, holds for
homeomorphisms $f:\Omega\to \Omega_*$ in metric spaces, see e.g.
Proposition 2.5 in \cite{RSa} or Proposition 13.5 in \cite{MRSY}.
Hence we have the following conclusion.

\medskip

\begin{proposition}\label{pro1}
{\it Let $\Omega$ and $\Omega_*$ be open sets on manifolds
$\mathbb{M}^{n}$ and $\mathbb{M}^{n}_*$, cor\-res\-pon\-ding\-ly,
and let $f:\Omega\to \Omega_*$ be a homeomorphism. Then
\begin{equation}\label{eq3}
C(p,f)\ \subseteq\ \partial \Omega_*\ \ \ \ \ \ \ \ \ \ \ \ \forall\
p\in\partial\Omega
\end{equation}}
\end{proposition}

In particular, we come from here to the following statement.

%\medskip

\begin{corollary}\label{cor1}
Let $D$ and $D_*$ be domains on Riemann surfaces $\mathbb{S}$ and
$\mathbb{S}_*$, correspondingly, and let $f:D\to D_*$ be a
homeomorphism. Then
\begin{equation}\label{eq3}
C(\partial D,f):\ =\bigcup\limits_{p\in \partial D} C(p,f)\
\subseteq\
\partial D_*
\end{equation}
\end{corollary}

%\medskip

Now, let us give the main result of the theory of uniformization of
Riemann surfaces that will be essentially applied later on, see e.g.  Section II.3 in
\cite{KAG}. The {\bf Poincare uniformization theorem} (1908)
states that every Riemann surface $\mathbb{S}$ is represented (up to
the conformal equivalence) in the form of the factor
$\widetilde{\mathbb{S}} \, / \, G$ where $\widetilde{\mathbb{S}}$ is
one of the canonical domains: $\overline{\mathbb{C}}$, $\mathbb{C}$
or the unit disk $\mathbb{D}$ in $\mathbb{C}$ and $G$ is a discrete
group of conformal ($=$ fractional) mappings of
$\widetilde{\mathbb{S}}$ onto itself. The corresponding Riemann
surfaces are called of {\bf elliptic, parabolic and hyperbolic
type}.

%\medskip

Moreover, $\widetilde{\mathbb{S}}=\overline{\mathbb{C}}$ only in the
case when $\mathbb{S}$ is itself conformally equivalent to the
sphere $\overline{\mathbb{C}}$ and the group $G$ is trivial, i.e.,
consists only of the identity mapping;
$\widetilde{\mathbb{S}}={\mathbb{C}}$ if $\mathbb{S}$ is conformally
equivalent to either ${\mathbb{C}}$, ${\mathbb{C}} \setminus \{ 0\}$
or a torus and, correspondingly, the group $G$ is either trivial or
is a group of shifts with one generator $z\to z+\omega$, $\omega\in\
{\mathbb{C}} \setminus \{ 0\}$ or a group of shifts with two
generators $z\to z+\omega_1$ and $z\to z+\omega_2$ where $\omega_1$
and $\omega_2\in\ {\mathbb{C}} \setminus \{ 0\}$ and $\ {\rm Im}\
\omega_1/\omega_2\ >\ 0$. Except these simplest cases, every Riemann
surface $\mathbb{S}$ is conformally equivalent to the unit disk
$\mathbb{D}$ factored by a discrete group $G$ without fixed points,
see e.g. Theorem 7.4.2 in \cite{ZVC}. And inversely, every factor
$\mathbb{D}/G$ is a Riemann surface, see e.g. Theorem 6.2.1
\cite{Be}.

%\medskip

In this connection, recall that we identify in the factor
$\widetilde{\mathbb{S}} \, / \, G$ all elements of the {\bf orbit}
$\ G_{z_0}:\ =\{\ z\in\ \widetilde{\mathbb{S}}:\ z=g(z_0),\ g\in G\
\}$ of every point $z_0\in\ \widetilde{\mathbb{S}}$. Recall also
that a group $G$ of fractional mappings of $\mathbb{D}$ onto itself
is called {\bf discrete} if the unit of $G$  (the identical mapping
$I$) is an isolated element of $G$. As easy to see, the latter
implies that all elements of the group $G$ are isolated each to
other. If the elements of the group $G$ have no fixed points as in
the uniformization theorem, then the latter is equivalent to that
the group $G$ {\bf discontinuously acts} on $\mathbb{D}$, i.e., for
every point $z\in\mathbb{D}$, there is its neighborhood $U$ such
that $g(U)\cap U=\varnothing$ for all $g\in G$, $g\ne I$, see e.g.
Theorem 8.4.1 in \cite{Be}.

%\medskip

Let us also describe in short the {\bf Poincare model} of
non-Euclidean plane, in other words, the so-called
Boyai-Gauss-Lobachevskii geometry or the hyperbolic geometry. Points
of the hyperbolic plane are points of the unit disk $\mathbb{D}$ and
{\bf hyperbolic straight lines} are the arcs in $\mathbb{D}$ of
circles that are perpendicular to the unit circle $\mathbb{S}^1:\ =\
\partial\mathbb{D}$ and the diameters of $\mathbb{D}$. Every two points in
$\mathbb{D}$ determine exactly a single hyperbolic straight line,
see e.g. Proposition 7.2.2 in \cite{ZVC}. The {\bf hyperbolic
distance} in the unit disk $\mathbb{D}$ is given by the formula
\begin{equation}\label{eqhyper}
h(z_1,z_2)\ =\ \log\ \frac{1+t}{1-t}\ ,\ \ \ \ \ \ \ {\mbox{where}}\
\ \ \ \ \ t\ =\ \frac{|z_1-z_2|}{|1-z_1\bar {z}_2|}\ ,
\end{equation}
the {\bf hyperbolic length} of a curve $\gamma$ and the {\bf
hyperbolic area} of a set $S$ in $\mathbb{D}$ are calculated as the
integrals, see e.g. Section 7.1 in \cite{Be}, Proposition 7.2.9 in
\cite{KAG},
\begin{equation}\label{eqlength}
s_h(\gamma)\ =\ \int\limits_{\gamma}\frac{2\, |dz|}{1-|z|^2}\ , \ \
\ \ \ \ \ h(S)\ =\ \int\limits_{S}\frac{4\, dx\, dy}{(1-|z|^2)^2}\
,\ \ \ \ \ \ \ {\mbox{where}}\ \ \ \  z = x + i y\, .
\end{equation}
All conformal ($=$ fractional) mappings of $\mathbb{D}$ onto itself
are  {\bf hyperbolic isometries}, i.e., they keep the hyperbolic
distance, see e.g. Theorem 7.4.1 in \cite{Be}, and hence the
hyperbolic length as well as the hyperbolic area are invariant under
such mappings.

%\medskip

A {\bf hyperbolic half-plane} $H$, i.e., one of two connected
components of the complement of a hyperbolic straight line $L$ in
$\mathbb{D}$, is a {\bf hyperbolically convex set}, i.e., every two
points in $H$ can be connected by a segment of a hyperbolic straight
line in $H$, see e.g. \cite{Be}, p. 128. A {\bf  hyperbolic polygon}
is a domain in $\mathbb{D}$ bounded by a Jordan curve, consisting of
segments of hyperbolic straght lines. If $G$ is a discrete group of
fractional mappings of $\mathbb{D}$ onto itself without fixed
points, then the {\bf Dirichlet polygon} for $G$ with the center
$\zeta\in\mathbb{D}$ is the convex set
\begin{equation}\label{eqHL} D_{\zeta}\ =\ \bigcap\limits_{g\in G,\ g\ne
I}\ H_g(\zeta)
\end{equation}
where
$$
H_g(\zeta) = \{ z\in\mathbb{D}:\ h(z,\zeta) < h(z,g(\zeta))\ \}
$$
is a hyperbolic half-plane containing the point $\zeta$ and bounded
by the hyperbolic straight line $L_g(\zeta) = \{ z\in\mathbb{D}:\
h(z,\zeta) = h(z,\, g(\zeta))\ \}$. $D_{\zeta}$ is also called the
{\bf Poincare polygon}. Dirichlet applied this construction at 1850
for the Euclidean spaces and, later on, Poincare has applied it to
hyperbolic spaces.

%\medskip

The geometric approach to the study of the factors $\mathbb{D} / G$
is based on the notion of its fundamental domains. A {\bf
fundamental set} for the group $G$ is a set $F$ in $\mathbb{D}$
containing precisely one point $z$ in every orbit $G_{z_0}$,
$z_0\in\mathbb{D}$. Thus, $\bigcup\limits_{g\in
G}g(F)=\mathbb{D}$. The existence of a fundamental set is guaranteed
by the choice axiom, see e.g. \cite{Wa}, p. 246. A domain $D\subset\mathbb{D}$
is called a {\bf fundamental domain} for $G$ if there is a fundamental set $F$ for
$G$ such that $D\subset F\subset\overline{D}$ and $h(\partial D)=0$.
If $D$ is a fundamental domain for a discrete group $G$ of fractional mappings $\mathbb{D}$
onto itself without fixed points, then $D$ and its images pave $\mathbb{D}$, i.e.,
\begin{equation}\label{eqfund}
\bigcup\limits_{g\in G}g(\overline{D})=\mathbb{D}\ , \ \ \ \ \ \ \
g(D)\cap D=\varnothing \ \ \ \ \ \ \ \forall\ g\in G,\ g\ne I\ .
\end{equation}
The Poincare polygon is an example of a fundamental domain that there is for every such a group,
see e.g. Theorem 9.4.2 in \cite{Be}.

%\medskip

The {\bf hyperbolic distance on a factor} $\mathbb{D}/G$ for a
discrete group $G$ without fixed points can be defined in the
following way. Let $p_1$ and $p_2\in\mathbb{D}/G$. Then by the
definition $p_1$ and $p_2$ are orbits $G_{z_1}$ and $G_{z_2}$ of
points $z_1$ and $z_2\in\mathbb{D}$. Set
\begin{equation}\label{eqdist}
h(p_1,\, p_2)\ =\ \inf\limits_{g_1, g_2\in G}\ h(\, g_1(z_1),\,
g_2(z_2)\, )\ .
\end{equation}
In view of discontinuous action of the group  $G$, no orbit have limit points inside of $\mathbb{D}$ and,
by the invariance of hyperbolic metric in $\mathbb{D}$ with respect to the group $G$, we have
\begin{equation}\label{eqequiv}
h\, (p_1\, ,\ p_2\, )\ =\ \min\limits_{g_1, g_2\in G}\ h\, (\,
g_1(z_1)\, ,\ g_2(z_2)\, )\ =
\end{equation}
$$
=\ \min\limits_{g\in G}\ h\, (\, z_1\ ,\ g(z_2)\, )\ =\
\min\limits_{g\in G}\ h\, (\, g(z_1)\, ,\ z_2\, )\ .
$$
It is easy to see from here that $h(p_1,p_2)=h(p_2,p_1)$ and that $h(p_1,p_2)\ne 0$
if $p_1\ne p_2$. It remains to show the triangle inequality. Indeed, let $p_0=G_{z_0}$, $p_1=G_{z_1}$ and $p_2=G_{z_2}$ and
let $h(p_0,p_1)=h(z_0, g_1(z_1))$ and $h(p_0,p_2)=h(z_0, g_2(z_2))$. Then we conclude from (\ref{eqequiv}) that $$h(p_1,p_2) \leq
h(g_1(z_1), g_2(z_2)) \leq h(z_0, g_1(z_1)) + h(z_0, g_2(z_2)) =
h(p_0,p_1) + h(p_0,p_2) .$$

Now, let $\pi :\ \mathbb{D}\to\mathbb{D}/G$ be the natural projection and let $F$
be a fundamental set in $\mathbb{D}$ for the group $G$. Let us consider in $F$ the metric
\begin{equation}\label{metric}
d(z_1, z_2):=h(\pi(z_1),\pi(z_2))\ .
\end{equation}
Note that by the construction $d(z_1, z_2)\leq h(z_1,z_2)$ and,
furthermore, $d(z_1, z_2)=h(z_1,z_2)$ if $z_2$ is close enough to
$z_1$ in the hyperbolic metric in $\mathbb{D}$. Thus, we obtain a
metric space $(F,\, d)$ that is homeomorphic to $\mathbb{D}/G$ where
the length and the area are calculated by the same formulas
(\ref{eqlength}). Note that the elements of the length and the area
in the integrals (\ref{eqlength})
\begin{equation}\label{eqelements}
ds_h(z)\ =\ \frac{2\, |dz|}{1-|z|^2}\ , \ \ \ \ \ \ \ dh(z)\ =\
\frac{4\, dx\, dy}{(1-|z|^2)^2}\ ,\ \ \ \ \ \ \ {\mbox{где}}\ \ \ \
z = x + i y\, ,
\end{equation}
are invariant with respect to fractional mappings of $\mathbb{D}$ onto itself, i.e.,
they are functions of the point $p\in\mathbb{D}/G$ and hence they make possible to calculate the length and the area on the
Riemann surfaces $\mathbb{D}/G$ with no respect to the choice of the fundamental set
$F$ and the corresponding local coordinates.

%\medskip

For visuality, later on we sometimes identify $\mathbb{D}/G$ with a
fundamental set $F$ in $\mathbb{D}$ for the group $G$ containing a
fundamental (Dirichlet-Poincare) domain for $G$. The factor
$\mathbb{D}/G$ has a natural complex structure for which the
projection $\pi :\ \mathbb{D}\to\mathbb{D}/G$ is a holomorphic
(single-valued analytic) function whose restriction to every
fundamental domain is a conformal mapping and, consequently, its
inverse mapping is a complex chart of the Riemann surface
$\mathbb{D}/G$.

It is clear that the distance (\ref{eqdist}), the elements of length
and area (\ref{eqelements}) do not depend on the choice of $G$ in
the Poincare uniformization theorem because they are invariant under
fractional mappings of $\mathbb{D}$ onto itself and we call them
{\bf hyperbolic} on the Riemann surface $\mathbb{S}$.

The case of a torus $\mathbb{S}$ is similar and much more simple,
and hence it is not separately discussed. In this case, we set
$s_h(z)=|dz|$ and $dh=dx\, dy$ but without the given name. The
latter elements of length and area are also invariant under the
corresponding (complex) proportional shifts  in the Poincare
uniformization theorem but up to the corresponding multiplicative
constants.

Given a family $\Gamma$ of paths $\gamma$ in $\mathbb{S}$, a Borel
function $\varrho:\mathbb{S}\to[0,\infty]$ is called {\bf
admissible} for $\Gamma$, abbr. $\varrho\in \mathrm{adm}\,\,\Gamma$,
if
\begin{equation}\label{eq13.2} \int\limits_{\gamma}\varrho(p)\,ds_h(p)\
\geq\ 1 \end{equation} for all $\gamma\in\Gamma$. The {\bf modulus}
of $\Gamma$ is given by the equality
\begin{equation}\label{eq13.5}M(\Gamma)\ =\ \inf\limits_{\varrho\in
\mathrm{adm}\,\Gamma}\int\limits_{\mathbb{S}} \varrho^{2}(p)\ dh(p)\
.\end{equation}

\bigskip

\section{On mappings  with finite distortion, the main lemma.} Recall that a homeomorphism $f$ between domains $D$ and $D^*$ in ${\Bbb R}^n$,
$n\geqslant2$, is called of {\bf finite distortion} if $f\in
W^{1,1}_{\rm loc}$ and
\begin{equation}\label{eqOS1.3} \Vert f'(x)\Vert^n\leqslant
K(x)\cdot J_f(x)\end{equation} with a function
$K$ that is a.e. finite. As usual, here $f'(x)$ denotes the Jacobian matrix of $f$ at $x \in D$ where it is determined,
$J_f(x)=\det f'(x)$ is the Jacobian of $f$ at $x$, and $\Vert
f'(x)\Vert$ is the operator norm of $f'(x)$, i.e.,
\begin{equation} \label{eq4.1.2} \Vert f'(x)\Vert =\max
\{|f'(x)h|:h\in{{\Bbb R}^n},|h|=1\}.\end{equation}

\medskip

First this notion was introduced in the plane for $f\in W^{1,2}_{\rm
loc}$ in the paper \cite{IS}. Later on, this condition was replaced by
$f\in W^{1,1}_{\rm loc}$, however, with the additional request
$J_f\in L^1_{\rm loc}$, see \cite{IM}. Note that the latter request can be omitted for homeomorphisms.
Indeed, for every homeomorphism
$f$ between domains $D$ and $D^*$ in ${\Bbb R}^n$ with first partial derivatives a.e. in $D$,
there is a set $E$ of the Lebesgue measure zero such that $f$ has $(N)-$property of Lusin on
$D\setminus E$ and
\begin{equation}\label{eqOS1.1.1}
\int\limits_{A}J_f(x)\,dm(x)=|f(A)|\end{equation} for every Borel set $A\subset D\setminus E$, see e.g. 3.1.4,
3.1.8 and 3.2.5 in \cite{Fe}.

%\medskip

In the complex plane, $\Vert f'\Vert=|f_z|+|f_{\overline{z}}|$ and
$J_f=|f_z|^2-|f_{\overline{z}}|^2$ where
$$f_{\overline{z}}=(f_x+if_y)/2\ ,\ \ \ f_{z}=(f_x-if_y)/2,\ \ \ z=x+iy\ ,$$
and $f_x$ and $f_y$ are partial derivatives of $f$ in $x$ and $y$, correspondingly.
Thus, in the case of sense-preserving homeomorphisms $f\in W^{1,1}_{\rm loc}$, (\ref{eqOS1.3}) is equivalent to
the condition that $K_f(z)<\infty $ a.e. where
\begin{equation} \label{eq4.1.4} K_f(z)=
\frac{|f_z|+|f_{\overline{z}}|}{|f_z|-|f_{\overline{z}}|}
\end{equation} if $|f_z|\neq |f_{\overline{z}}|$, $1$ if $f_z=0=f_{\overline{z}}$, and $\infty$ in the rest cases.
As usual, the quantity $K_{f}(z)$ is called {\bf dilatation} of the mapping $f$ at $z$.

%\medskip

If $f:D\to D^*$ is a homeomorphism of the class $W^{1,1}_{\rm loc}$ between
domains $D$ and $D^*$ on the Riemann surfaces $\mathbb{S}$ and
$\mathbb{S}^*$, then $K_{f}(z)$ denotes the dilatation of the mapping $f$ in
local coordinates, i.e., the dilatation of the mapping $F$ in
(\ref{eq2}). The geometric sense of the quantity (\ref{eq4.1.4}) at a point $z$
of differentiability of the mapping $f$ is the ratio of half-axes of the infinitesimal ellipse
into which the infinitesimal circle centered at the point is transferred under the mapping $f$.
The given quantity is invariant under the replacement of local coordinates, because
conformal mappings transfer infinitesimal circles into infinitesimal circles and
infinitesimal ellipses into infinitesimal ellipses with the same ratio of half-axes, i.e.,
$K_{f}$ is really a function of a point $p\in\mathbb{S}$ but not of local coordinates.

%\medskip

We will call a homeomorphism $f:D\to D^*$ between domains $D$ and
$D^*$ on Riemann surfaces $\mathbb{S}$ and $\mathbb{S}^*$ by a {\bf
mapping with finite distortion} if $f$ is so in local coordinates.
It is clear that this property enough to verify only for one atlas
because conformal mappings have $(N)-$property of Lusin. We will say
also that a homeomorphism $f:D\to D^*$ between domains $D$ and $D^*$
in the compactifications of Kerekjarto-Stoilow
$\overline{\mathbb{S}}$ and $\overline{\mathbb{S}^*}$ is a mapping
with finite distortion if this property holds for its restriction to
$\mathbb{S}$. Note that a homeomorphism between domains in
$\mathbb{S}$ and $\mathbb{S}^*$ is always extended to a
homeomorphisms between the corresponding domains in
$\overline{\mathbb{S}}$ and $\overline{\mathbb{S}^*}$. Later on, we
assume that $K_{f}$ is extended by zero outside of $D$ and write $\
K_{f}\in L^{1}_{\rm loc}$ if $K_{f}$ is locally integrable with
respect to the area $h$ on $\mathbb{S}$.

\medskip

\begin{lemma}\label{lem1} Let $D$ and $D^*$ be domains on Riemann surfaces
$\mathbb{S}$ and $\mathbb{S}^*$. If $f:D\to D^*$ is a homeomorphism
of finite distortion  with $\ K_{f}\in L^{1}_{\rm loc}$, then
\begin{equation}\label{eqOS1.8a}M\left(\Delta\left(fC_1,fC_2;fA\right)\right)\
\leqslant\ \int\limits_{A}K_f(p)\cdot\ \xi^2(h(p,p_0))\ dh(p)\ \ \ \
\ \ \ \ \ \forall\ p_0\in\ \overline D\end{equation} for every ring
$A=A(p_0,R_1,R_2)=\{ p\in\mathbb{S}: R_1<h(p,p_0)<R_2 \}$, the
circles $C_1=\{ p\in\mathbb{S}:\, h(p,p_0)=r_1 \}$, $\ C_2=\{
p\in\mathbb{S}:\, h(p,p_0)=r_2 \}$,
$0<R_1<R_2<\varepsilon=\varepsilon(p_0)$, and every measurable
function $\xi:(R_1,R_2)\to[0,\infty]$ such that
\begin{equation}\label{eqOS1.9}\int\limits_{R_1}^{R_2}\xi(R)\ dR\geqslant\ 1\ .\end{equation}
\end{lemma}

\begin{proof} As it was discussed in Section 2, here we identify the Riemann surface  $\ \mathbb{D}/G$ with
a fundamental set $F$ in $\, \mathbb{D}$ for $G$ with the metric $d$
defined by (\ref{metric}) that contains a fundamental polygon of
Poincare $D_{z_0}$ for $G$ centered at a point $z_0\in\, \mathbb{D}$
whose orbit $G_{z_0}$ is $p_0$. With no loss of generality we may
assume that $z_0=0$. The latter always can be obtained with the help
of the fractional mapping of $\mathbb{D}$ onto itself
$g_0(z)=(z-z_0)/(1-z\overline{z_0})$ transfering the point $z_0$
into the origin. Passing to the new group $G_0$ we obtain the
Riemann surface $\ \mathbb{D}/G_0$ that is conformally equivalent to
$\ \mathbb{D}/G$. Set
$$\delta_0\ =\ \min\ \left[\, \inf\limits_{\zeta\in \partial D_{0}}\, d(0,\, \zeta),\
\sup\limits_{z\in D}\, d(0,\, z)\, \right]\ . $$ Let us choice $\delta\in(0,\delta_0)$ so small that, for
$d(0,z)\leqslant\delta$, the equality $d(0,z)=h(0,z)$ holds. Note that correspondingly to (\ref{eqhyper})
$$ R\ :\, =\ h(0,z)\
=\ \log\ \frac{1+r}{1-r}\ ,\ \ \ \ \ \ \ {\mbox{where}}\ \ \ \ \ \ r\
:\, =\ |z|\ ,
$$
and, correspondingly,
$$
dR\ =\ \frac{2dr}{1-r^2} \ ,\ \ \ \ \ \ \ r\ =\ \frac{e^R-1}{e^R+1}\
.
$$
Consequently,
$$
\int\limits_{r_1}^{r_2}\eta(r)\ dr\geqslant\ 1
$$
where
$$
\eta(r)\ =\ \frac{2}{1-r^2}\cdot\ \xi\left( \log\
\frac{1+r}{1-r}\right)
$$
and, moreover,
\begin{equation}\label{eqeq} \int\limits_{A}K_f(z)\cdot\ \xi^2(d(z,z_0))\
dh(z)\ =\ \int\limits_{A}K_f(z)\cdot\ \eta^2(|z|)\ dm(z)
\end{equation}
where the element of the area $dm(z):\, =dx\, dy$ corresponds to the Lebesgue measure in the plane $\mathbb{C}$.
Moreover, note that  $A=\{ z\in\mathbb{D}:
r_1<|z|<r_2 \}$, $C_1=\{ z\in\mathbb{D}: |z|=r_1 \}$ и $\ C_2=\{
z\in\mathbb{D}: |z|=r_2 \}$.

It is clear that the subset of the complex plane $D(\delta):\, =\{
z\in D : |z|<\delta\}$ is decomposed into at most a countable
collection of domains. Then components of the set $f(D(\delta))$ are
homeomorphic to these domains and, consequently, by the general
principle of Koebe, see e.g.  Section II.3 in \cite{KAG}, they are
conformally equivalent to plane domains, i.e., the family of curves
$\Delta(fC_1,fC_2;fA)$ is decomposed into a countable collection of
its subfamilies, belonging to the corresponding mutually disjoint
complex charts of the Riemann surface $\ \mathbb{D}/G^*$. Thus, the
conclusion of our lemma follows from Theorem 3 in \cite{KPRS}.
\end{proof} $\Box$

\begin{remark}\label{rmk1}
{ \rm In other words, the statement of Lemma \ref{lem1} means that
every homeomorphism $f$ of finite distortion between domains on
Riemann surfaces with $\ K_{f}\in L^{1}_{\rm loc}$ is the so-called
ring $Q-$homeomorphism with $Q=K_f$. Note also that Riemann surfaces
are locally the so-called Ahlfors $2-$regular spaces with the
mentioned metric and measure $h$, see e.g. Theorem 7.2.2 in
\cite{Be}. Hence further we may apply results of the paper
\cite{Smol} on the boundary behavior of ring $Q-$homeomorphisms in
metric spaces to homeomorphisms with finite distortion between
domains on Riemann surfaces. It makes possible us, in comparison
with the papers \cite{VR} and \cite{RV}, to formulate new results in
terms of the metric and measure $h$ but not in terms of local
coordinates on Riemann surfaces. Recall that the boundary behavior
of Sobolev's homeomorphisms on smooth Riemannian manifolds for
$n\geq 3$ was investigated in the paper \cite{ARS}.}
\end{remark}

\bigskip

\section{On weakly flat and strongly accessible boundaries}

In this section, we follow paper \cite{RSa}, see also Chapter 13 in
monograph \cite{MRSY}.

\medskip

Later on, given sets $E, F$ and $\Omega$ in a Riemann surface
$\mathbb{S}$, $\Delta(E,F;\Omega)$ denotes the family of all curves
$\gamma:[a,b]\to\mathbb{S}$ that join the sets $E$ and $F$ in
$\Omega$, i.e., $\gamma(a)\in E$, $\gamma(b)\in F$ and $\gamma(t)\in
\Omega$ for $a<t<b$.

%\medskip

It is said that the boundary of a domain $D$ in $\mathbb{S}$ is {\bf
weakly flat at a point} $z_0\in\partial D$ if, for every
neighborhood $U$ of the point $z_0$ and every number $N>0$, there is
a neighborhood $V\subset U$ of the point $z_0$ such that
\begin{equation}\label{eq1.5KR}
M\left(\Delta\left(E,F;D\right)\right)\ \geqslant\ N\end{equation}
for all continua $E$ and $F$ in $D$ intersecting $\partial U$ and
$\partial V$. The boundary of $D$ is called {\bf weakly flat} if it
is weakly flat at every point in $\partial D$. Note that smooth and
Lipshitz boundaries are weakly flat.

%\medskip

It is also said that a point $z_0\in\partial D$ is {\bf strongly
accessible} if, for every neighborhood $U$ of the point $z_0$ there
exist a continuum $E$ in $D$, a neighborhood $V\subset U$ of the
point $z_0$ and a number $\delta>0$ such that
\begin{equation}\label{eq1.6KR}
M\left(\Delta\left(E,F;D\right)\right)\ \geqslant\
\delta\end{equation} for every continuum $F$ in $D$ intersecting
$\partial U$ and $\partial V$. The boundary of $D$ is called {\bf
 strongly accessible} if every point $z_0\in\partial D$ is so.

%\medskip

It is easy to see that if the boundary of a domain $D$ in
$\mathbb{S}$ is weakly flat at a point $z_0\in\partial D$, then the
point $z_0$ is strongly accessible from $D$. Moreover, it was proved
in metric spaces with measures that if a domain $D$ is weakly flat
at a point $z_0\in\partial D$, then $D$ is locally connected at
$z_0$, see e.g. Lemma 3.1 in \cite{RSa} or Lemma 13.1 in
\cite{MRSY}.

\medskip

\begin{proposition}\label{pro2}
{\it If a domain $D$ on a Riemann surface ${\mathbb S}$ is weakly
flat at a point in $\partial D$, then $D$ is locally connected at
the point.}
\end{proposition}

Recall that a domain $D$ is called {\bf locally connected at a
point} in $\partial D$ if, for every neighborhood $U$ of the point,
there is its neighborhood $V\subseteq U$ such that $V\cap D$ is a
domain.

\bigskip

\section{On extending to the boundary of the inverse mappings} In contrast with the direct mappings,
see the next section, we have the following simple criterion for the
inverse mappings.

\medskip

\begin{theorem}\label{th1}
Let $\mathbb{S}$ and $\ \mathbb{S}^*$ be Riemann surfaces, $D$ and
$D^*$ be domains in $\, \overline{\mathbb{S}}$ and $\,
\overline{\mathbb{S}^*}$, correspondingly, $\partial
D\subset\mathbb{S}$ and $\ \partial D^*\subset\mathbb{S}^*$, $D$ be
locally connected on its boundary and let $\partial D^*$ be weakly
flat. Suppose that $f:D\to D^*$ is a homeomorphism of finite
distortion with $\ K_{f}\in L^{1}_{\rm loc}$. Then the inverse
mapping $g=f^{-1}:D^*\to D$ can be extended by continuity to a
mapping $g:\overline{D^*}\to\overline{D}$.
\end{theorem}

%\medskip

As it was before, we assume here that the dilatation $K_{f}$ is extended by zero outside of the domain $D$.

\medskip

\begin{proof} By the Uryson theorem, see e.g. Theorem 22.II.1 in \cite{Ku$_1$},  $\overline{\mathbb{S}}$
is a metrizable space. Hence the compactness of
$\overline{\mathbb{S}}$  is equivalent to its sequential
compactness, see e.g. Remark 41.I.3 in \cite{Ku$_2$}, and the
closure $\overline{D}$ is a compact subset of ${\mathbb{S}}$, see
e.g. Proposition I.9.3 in \cite{Bou}. Thus, the conclusion of
Theorem \ref{th1} follows by Theorem 5 in \cite{Smol} as well as by
Lemma \ref{lem1} and Remark \ref{rmk1}.
\end{proof} $\Box$

\bigskip

\section{On extending to the boundary of the direct mappings} As it was before,
we assume here that the function $K_f$ is extended by zero outside
of the domain $D$.

In contrast to the case of the inverse mappings, as it was already
established in the plane, no degree of integrability of the
dilatation leads to the extension to the boundary of direct mappings
of the Sobolev class, see e.g. the proof of Proposition 6.3 in
\cite{MRSY}. The corresponding criterion for that given below is
much more refined. Namely, in view of Lemma \ref{lem1} and Remark
\ref{rmk1}, by Lemma 3 in \cite{Smol} we obtain the following.

\medskip

\begin{lemma}\label{lem3}
Let $\mathbb{S}$ and $\, \mathbb{S}^*$ be Riemann surfaces, $D$ and
$\, D^*$ be domains in $\overline{\mathbb{S}}$ and $
\overline{\mathbb{S}^*}$, correspondingly, $\partial
D\subset\mathbb{S}$, $\partial D^*\subset\mathbb{S}^*$, $D$ be
locally connected at a point $p_0\in\partial D$. Suppose that
$f:D\to D^*$  is a homeomorphism of finite distortion with $\
K_{f}\,\in\, L^{1}_{\rm loc}\ $ and $\
\partial D^*$ is strongly accessible at least at one point in the cluster set $\ C(p_0,f)\,$
and
\begin{equation}\label{eqpsi} \int\limits_{\varepsilon<h(p,p_0)<\varepsilon_0}
K_f(p)\cdot\,\psi^{2}_{p_0,\varepsilon}(h(p,p_0))\ dh(p)\ =\
o(I_{p_0,\varepsilon_0}^{2}(\varepsilon))\ \ \ \ \ \ \ \ \mbox{as}\
\ \  \  \varepsilon\ \to\ 0\end{equation} for some $\varepsilon_0>0$
where $\psi_{p_0,\varepsilon}(t)$ is a family of nonnegative
measurable (by Lebesgue) functions on $(0,\infty)$ such that
\begin{equation}\label{eqI} 0\ <\ I_{p_0,\varepsilon_0}(\varepsilon)\colon=\int\limits_{\varepsilon}^{\varepsilon
_0} \psi_{p_0,\varepsilon}(t)\,dt\ <\ \infty\qquad\forall\
\varepsilon\in(0,\varepsilon_0)\ .\end{equation}

Then $\, f$ is extended by continuity to the point $\, p_{0}$ and
$\, f(p_0)\,\in\
\partial D^*$.
\end{lemma}

\medskip

Note that conditions (\ref{eqpsi})-(\ref{eqI}) imply that
$I_{p_0,\varepsilon_0}(\varepsilon)\to\infty$ as $\varepsilon\to 0$
and that $\varepsilon_0$ can be chosen arbitrarily small with
keeping (\ref{eqpsi})-(\ref{eqI}).

\medskip

Lemma \ref{lem3} makes possible to obtain a series of criteria on the continuous
extension to the boundary of mappings with finite distortion between domains
on Riemann surfaces. Here we assume that $K_f\equiv 0$ outside of $D$.

\medskip

\begin{theorem}\label{th2}
{Let $\, \mathbb{S}$ and $\, \mathbb{S}^*$ be Riemann surfaces, $D$
and $D^*$ be domains on $\, \overline{\mathbb{S}}$ and $\,
\overline{\mathbb{S}^*}$, correspondingly, $\partial
D\subset\mathbb{S}$ and $\ \partial D^*\subset\mathbb{S}^*$, $D$ be
locally connected on its boundary and $\partial D^*$ be strongly
accessible. Suppose that $f:D\to D^*$ is a homeomorphism of finite
distortion with $\, K_{f}\in L^{1}_{\rm loc}$ and
\begin{equation}\label{eq8.11.2}\int\limits_{0}^{\delta}
\frac{dr}{||K_{f}||\, (p_0,r)}\ =\ \infty\ \ \ \ \ \ \ \ \forall\
p_0\in\partial D
\end{equation} where
\begin{equation}\label{eq8.11.4}
||K_{f}||\, (p_0,r)\ = \int\limits_{h(p,p_0)=r}K_{f}(p)\ ds_h(p)\
.\end{equation}

Then the mapping $f$ is extended by continuity to $\, \overline{D}$
and $\, f(\partial D)= \partial D^*$.}
\end{theorem}

\medskip

\begin{proof} Indeed, setting $\psi_{p_0}(t)=1/||K_{f}||\,
(p_0,t)$ for all $t\in(0,\varepsilon_0)$ under small enough
$\varepsilon_0>0$ and $\psi_{p_0}(t)=1$ for all
$t\in(\varepsilon_0,\infty)$, we obtain from condition
(\ref{eq8.11.2}) that
$$
\int\limits_{\varepsilon<h(p,p_0)<\varepsilon_0}
K_f(p)\cdot\,\psi^{2}_{p_0}(h(p,p_0))\ d\, h(p)\ =\
I_{p_0,\varepsilon_0}(\varepsilon)\ =\
o(I^2_{p_0,\varepsilon_0}(\varepsilon))\ \ \ \ \ \ \mbox{as}\ \ \
\varepsilon\to 0
$$
where, in view of the conditions $\ K_f(p)\geqslant 1\ $ in $\ D\ $
and $\ K_{f}\,\in\, L^{1}_{\rm loc}\ $,
$$
0\ <\
I_{p_0,\varepsilon_0}(\varepsilon):\,=\int\limits_{\varepsilon}^{\varepsilon
_0} \psi_{p_0}(t)\,dt\ <\ \infty\ .
$$
Thus, the first conclusion of Theorem \ref{th2} follows from Lemma
\ref{lem3}. The se\-cond conclusion of Theorem \ref{th2} follows
e.g. from Proposition 2.5 in \cite{RSa}, see also  Proposition 13.5
in \cite{MRSY}.
\end{proof} $\Box$

\medskip

\begin{corollary} {\it In particular, the conclusion of Theorem \ref{th2} holds if
\begin{equation}\label{eqhLOG}
K_f(p)\ \ =O{\left(\log\frac{1}{h(p,p_0)}\right)} \ \ \ \ \ \
\mbox{as}\ \ \ p\to p_0\ \ \ \ \ \ \ \ \forall\ p_0\in\partial D
\end{equation}
or, more generally,
\begin{equation}\label{eqlog}
k_{p_{0}}(\varepsilon)=O{\left(\log\frac{1}{\varepsilon}\right)} \ \
\ \ \ \ \mbox{as}\ \ \ \varepsilon\to 0\ \ \ \ \ \ \ \ \forall\
p_0\in\partial D
\end{equation} where $k_{p_{0}}(\varepsilon)$ is the mean value of the function $K_{f}$ over the circle $h(p,p_0)=\varepsilon$.}
\end{corollary}

\medskip

By Theorem 3.1 in \cite{RSY} with $\lambda_2=e/\pi$ we have the
following consequence from Theorem \ref{th2}, see also arguments in
the proof of Lemma \ref{lem1}.

\medskip

\begin{theorem} \label{th3} {\it Under hypotheses of Theorem \ref{th2},
suppose that
\begin{equation}\label{eq1009}\int\limits_{U}\Phi(K_{f}(p))\ dh(p)\ <\
\infty\end{equation}  in a neighborhood $U$ of $\partial D$ where
$\Phi:\overline{\Bbb R}_{+}\to{\overline{\Bbb R}_{+}} $ is a
nondecreasing convex function with the condition
\begin{equation}\label{eq1010}\int\limits_{\delta}^{\infty}\frac{d\tau}{\tau\Phi^{-1}(\tau)}=\infty\ ,\ \ \ \ \ \ \ \ \ \ \delta>\Phi(0)\ .\end{equation}

Then the mapping $f$ is extended by continuity to $\, \overline{D}$
and $\, f(\partial D)= \partial {D^*}$.}
\end{theorem}

%\medskip

\begin{remark}
{ \rm Note by Theorem 5.1 and Remark 5.1 in \cite{KR$_3$} condition (\ref{eq1010})
is not only necessary but also sufficient for the continuous extension to the boundary
of all mappings $f$ of finite distortion with integral restrictions of the form
(\ref{eq1009}). Note also that by Theorem 2.1 in \cite{RSY} condition (\ref{eq1010})
is equivalent to each of the following conditions where $H(t)=\log\Phi(t)$:
\begin{equation}\label{eq1012}\int\limits_{\Delta}^{\infty}H'(t)\,\frac{dt}{t}=\infty\ ,\end{equation}
\begin{equation}\label{eq1013}\int\limits_{\Delta}^{\infty}
\frac{dH(t)}{t}=\infty\ ,\end{equation}
\begin{equation}\label{eq1014}\int\limits_{\Delta}^{\infty}H(t)\,\frac{dt}{t^2}=\infty\end{equation}
for some $\Delta>0$, and also to each of the equality:
\begin{equation}\label{eq1015}\int\limits_{0}^{\delta}H\left(\frac{1}{t}\right)\,{dt}=\infty\end{equation}
for some $\delta>0$,
\begin{equation}\label{eq1016}\int\limits_{\Delta_*}^{\infty}
\frac{d\eta}{H^{-1}(\eta)}=\infty\end{equation} for some
$\Delta_*>H(+0)$.

\medskip

Here the integral in (\ref{eq1013}) is understood as the Lebesgue-Stiltjes integral,
and the integrals in (\ref{eq1012}),
(\ref{eq1014})--(\ref{eq1016}) as the usual Lebesgue integrals.

\medskip

It is necessary to give more explanations. In the right hand sides of conditions
(\ref{eq1012})--(\ref{eq1016}), we have in mind $+\infty$. If
$\Phi(t)=0$ for $t\in[0,t_*]$, then $H(t)=-\infty$ for $t\in[0,t_*]$,
and we complete the definition in (\ref{eq1012}) setting $H'(t)=0$ for
$t\in[0,t_*]$. Note that conditions (\ref{eq1013}) and (\ref{eq1014}) exclude
that $t_*$ belongs to the interval of integrability because in the contrary case
the left hand sides in (\ref{eq1013}) and (\ref{eq1014}) either are equal $-\infty$
or not determined. Hence we may assume that in (\ref{eq1012}--(\ref{eq1015}) $\delta>t_0$,
correspondingly, $\Delta<1/t_0$ where $t_0:=\sup_{\Phi(t)=0}t$ and $t_0=0$ if $\Phi(0)>0$.

\medskip

Among the conditions counted above, the most interesting one is condition (\ref{eq1014})
that can be written in the form:
\begin{equation}\label{eq5!} \int\limits_{\delta}^{\infty}\log \Phi(t)\ \frac{dt}{t^{2}}\ =\ \infty\ .\end{equation}
}
\end{remark}

%\medskip

\begin{corollary} {\it In particular, the conclusion of Theorem \ref{th3} holds if, for some $\alpha>0$},
\begin{equation}\label{eq111}
\int\limits_{U} e^{\alpha K_{f}(p)}\ dh(p)\ <\ \infty\
.\end{equation}
\end{corollary}

%\medskip

The following statement follows from Lemma \ref{lem3} for
$\psi(t)=1/t$.

\medskip

\begin{theorem} \label{th4} {\it Under the hypotheses of Theorem
\ref{th2}, if
\begin{equation}\label{eqOSKRSS10.336a}\int\limits_{\varepsilon<h(p,p_0)<\varepsilon_0}K_f(p)\ \frac{dh(p)}{h(p,p_0)^2}\
=\
o\left(\left[\log\frac{1}{\varepsilon}\right]^2\right)\quad\mbox{as}\
\ \  \varepsilon\to 0 \ \ \ \forall\ p_0\in\partial D\
,\end{equation} then the mapping $f$ is extended by continuity to
$\, \overline{D}$ and $\, f(\partial D)= \partial {D^*}$.}
\end{theorem}

%\medskip

\begin{remark}
{\rm Choosing in Lemma \ref{lem3} the function $\psi(t)=1/(t\log 1/t)$ instead of $\psi(t)=1/t$,
we obtain that condition (\ref{eqOSKRSS10.336a}) can be replaced by the conditions
\begin{equation}\label{eqOSKRSS10.336b}
\int\limits_{\varepsilon<h(p,p_0)<\varepsilon_0}\frac{K_f(p)\ dh(p)}
{\left(h(p,p_0)\ \log\frac{1}{h(p,p_0)}\right)^2}\ =\
o\left(\left[\log\log\frac{1}{\varepsilon}\right]^2\right)\quad\mbox{as}\
\ \  \varepsilon\to 0\ .\end{equation} Similarly, condition
(\ref{eqlog}) by Theorem \ref{th2} can be replaced by the weaker
condition
\begin{equation}\label{eqOSKRSS10.336h} k_{z_0}(\varepsilon)\ =\ O\left(\log\frac{1}{\varepsilon}\log\,\log\frac{1}{\varepsilon}\right)\quad\mbox{as}\
\ \  \varepsilon\to 0\ .\end{equation} Of course, we could give here a series
of the corresponding conditions of the logarithmic type applying suitable functions $\psi(t)$. }
\end{remark}

%\medskip

Following paper \cite{RSa}, cf. \cite{IR}, see also Section 13.4 in
\cite{MRSY}, Section 2.3 in \cite{GRSY}, we say that a function
$\varphi:\mathbb{S}\to{\Bbb R}$ has {\bf finite mean oscillation} at
a point $p_0\in \mathbb{S}$, written $\varphi\in{\rm FMO}(p_0)$, if
\begin{equation}\label{eqFMO}\limsup\limits_{\varepsilon\to 0}\ \dashint_{B(p_0,\,\varepsilon)}|\
\varphi(p)-\widetilde{ \varphi}_{\varepsilon}|\ dh(p)\ <\
\infty\end{equation} where $\widetilde{{ \varphi}}_{\varepsilon}$ is
the mean value of $\varphi$ over the disk
$B(p_0,\,\varepsilon)=\{p\in\mathbb{S}:\ h(p,p_0)<\varepsilon\}$.

\medskip

By Remark \ref{rmk1} and Lemma \ref{lem3} with the choice
$\psi_{p_0,\,\varepsilon}(t)\equiv 1/t\log\frac{1}{t}$, in view of
Lemma 4.1 and Remark 4.1 in \cite{RSa}, see also Lemma 13.2 and
Remark 13.3 in \cite{MRSY}, we obtain the following result.

\medskip

\begin{theorem} \label{th5} {\it If under the hypotheses of Theorem \ref{th2}, for
some $Q:\mathbb{S}\to{\Bbb R}^+$,
\begin{equation}\label{eq1007}{K_{f}(p)\ \leqslant\ Q(p)\in\ {\rm FMO}(p_0)}\ \ \ \ \ \ \ \ \forall\ p_0\in\partial D\
.\end{equation}

Then the mapping $f$ is extended by continuity to $\, \overline{D}$
and $\, f(\partial D)= \partial {D^*}$.}
\end{theorem}

\medskip

By Corollary 4.1 in \cite{RSa}, see also Corollary 13.3 in
\cite{MRSY}, we have also from Theorem \ref{th5} the next statement:

\medskip

\begin{corollary} {\it In particular, the conclusion of Theorem \ref{th5} holds if
\begin{equation}\label{eq111}\limsup\limits_{\varepsilon\to 0}\
\dashint_{B(p_0,\,\varepsilon)}K_{f}(p)\ dh(p)\ <\ \infty\ \ \ \ \ \
\ \ \forall\ p_0\in\partial D\ .\end{equation}}
\end{corollary}

\medskip

\begin{remark}
{\rm Note that Lemma \ref{lem3} makes possible also to realize the point-wise analysis:
if the given conditions for the dilatation hold at one boundary point of $D$, then
the extension of the mappings by continuity holds at this point. However, not to be repeated
we will not formulate here the corresponding point-wise results in the explicit form.}
\end{remark}

\bigskip

\section{On homeomorphic extension to the boundary}

Combining Theorem \ref{th1} and results of the last section, we obtain a series
of effective criteria of the homeomorphic extension to the boundary of the mappings
with finite distortion between domains on Riemann surfaces. As it was before,
here we assume that the function $K_f$ is extended by zero outside of the domain $D$.

\medskip

\begin{theorem} \label{th6} {\it Let under the hypotheses of Theorem \ref{th1}
\begin{equation}\label{eqh8.11.2}\int\limits_{0}^{\delta}
\frac{dr}{||K_{f}||\, (p_0,r)}\ =\ \infty\ \ \ \ \ \ \ \ \forall\
p_0\in\partial D\end{equation} where
\begin{equation}\label{eqh8.11.4}
||K_{f}||\, (p_0,r)\ = \int\limits_{h(p,p_0)=r}K_{f}(p)\ ds_h(p)\
.\end{equation}

Then the mapping $f$ is extended to the homeomorphism of $\
\overline{D}$ onto $\ \overline{D^*}$.}
\end{theorem}

\medskip

\begin{corollary} {\it In particular, the conclusion of Theorem \ref{th6} holds if
\begin{equation}\label{eqhLOG}
K_f(p)\ \ =O{\left(\log\frac{1}{h(p,p_0)}\right)} \ \ \ \ \ \
\mbox{as}\ \ \ p\to p_0\ \ \ \ \ \ \ \ \forall\ p_0\in\partial D
\end{equation}
or, more generally,
\begin{equation}\label{eqhlog}
k_{p_{0}}(\varepsilon)=O{\left(\log\frac{1}{\varepsilon}\right)} \ \
\ \ \ \ \mbox{as}\ \ \ \varepsilon\to 0\ \ \ \ \ \ \ \ \forall\
p_0\in\partial D
\end{equation} where $k_{p_{0}}(\varepsilon)$ is the mean value of the function $K_{f}$ over the circle $
h(p,p_0)=\varepsilon$.}
\end{corollary}

\medskip

\begin{theorem} \label{th7} {\it Under the hypotheses of Theorem \ref{th1}, suppose that
\begin{equation}\label{eqh1009}\int\limits_U\Phi(K_{f}(p))\ dh(p)\ <\ \infty\end{equation}
in a neighborhood $U$ of $\partial D$ where $\Phi:\overline{\Bbb
R}_{+}\to{\overline{\Bbb R}_{+}} $ is a nondecreasing convex
function with the condition
\begin{equation}\label{eqh1010}\int\limits_{\delta}^{\infty}\frac{d\tau}{\tau\Phi^{-1}(\tau)}=\infty\end{equation}
for some $\delta>\Phi(0)$. Then the mapping $f$ is extended to a homeomorphism of $\ \overline{D}$ onto $\ \overline{D^*}$.}
\end{theorem}

\begin{corollary} {\it In particular, the conclusion of Theorem \ref{th7} holds if, for some
$\alpha>0$}, in a neighborhood $U$ of $\partial D$
\begin{equation}\label{eqh1011}
\int\limits_U e^{\alpha K_{f}(p)}\ dh(p)\ <\ \infty\ .\end{equation}
\end{corollary}

\begin{theorem} \label{th8} {\it Let under the hypotheses of Theorem \ref{th1}
\begin{equation}\label{eqhOSKRSS10.336a}\int\limits_{\varepsilon<h(p,p_0)<\varepsilon_0}K_f(p)\ \frac{dh(p)}{h(p,p_0)^2}\
=\
o\left(\left[\log\frac{1}{\varepsilon}\right]^2\right)\quad\mbox{as}\
\ \  \varepsilon\to 0\ \ \ \ \ \ \ \ \forall\ p_0\in\partial D\
.\end{equation}

Then the mapping $f$ is extended to a homeomorphism of $\
\overline{D}$ onto $\ \overline{D^*}$.}
\end{theorem}

\medskip

\begin{theorem} \label{th9} {\it Let under the hypotheses of Theorem
\ref{th1}, for some $Q:\mathbb{S}\to\mathbb{R}^+$,
\begin{equation}\label{eqh1007}{K_{f}(p)\ \leqslant\ Q(p)\in\ {\rm FMO}(p_0)}\ \ \ \ \ \ \ \ \forall\
p_0\in\partial D\ .\end{equation} Then the mapping $f$ is extended
to a homeomorphism of $\ \overline{D}$ onto $\ \overline{D^*}$.}
\end{theorem}

%\medskip

\begin{corollary} {\it In particular, the conclusion of Theorem \ref{th9} holds if
\begin{equation}\label{eqh111}\limsup\limits_{\varepsilon\to 0}\
\dashint_{B(p_0,\,\varepsilon)}K_{f}(p)\ dh(p)\ <\ \infty\ \ \ \ \ \
\ \ \forall\ p_0\in\partial D\ .\end{equation}}
\end{corollary}

\bigskip

%\newpage

\noindent
{\bf Vladimir Ryazanov, Sergei Volkov,}\\
Institute of Applied Mathematics and Mechanics\\
of National Academy of Sciences of Ukraine,\\
UKRAINE, 84116, Slavyansk, 19 General Batyuk Str.,\\
vl$\underline{\ \ }$\,ryazanov@mail.ru, sergey.v.volkov@mail.ru

\end{document}